\documentclass[12pt,
reqno]{amsart}
\input epsf.tex
\usepackage{graphics}
\usepackage{graphicx}
\usepackage{amsmath,amssymb,amscd,amsfonts,amsthm}
\usepackage{graphics}

\usepackage{amsmath,amssymb,amscd,amsfonts,amsthm}
\usepackage{graphics}
\usepackage{dsfont}
\usepackage[all]{xy}
\usepackage{enumerate}

\usepackage{mathrsfs}

\usepackage{array,amscd,latexsym, verbatim}
\usepackage{enumerate}
\usepackage[all]{xy}
\usepackage{paralist}

{\catcode`\@=11
\gdef\n@te#1#2{\leavevmode\vadjust{%
 {\setbox\z@\hbox to\z@{\strut#1}%
  \setbox\z@\hbox{\raise\dp\strutbox\box\z@}\ht\z@=\z@\dp\z@=\z@%
  #2\box\z@}}}
\gdef\leftnote#1{\n@te{\hss#1\quad}{}}
\gdef\rightnote#1{\n@te{\quad\kern-\leftskip#1\hss}{\moveright\hsize}}
\gdef\?{\FN@\qumark}
\gdef\qumark{\ifx\next"\DN@"##1"{\leftnote{\rm##1}}\else
 \DN@{\leftnote{\rm??}}\fi{\rm??}\next@}}

\theoremstyle{plain}

\newtheorem{theorem}{Theorem}

\newtheorem{proposition}{Proposition}
\newtheorem{lemma}{Lemma}
\newtheorem{remark}{\bf Remark}

\theoremstyle{definition}

\newtheorem{definition}{Definition}

\newtheorem{nothing*}[theorem]{}
\newtheorem{subnothing*}[sub]{}

\newtheorem*{exsnonumber}{Examples}

\theoremstyle{remark}

\def\Aut {{\rm Aut\,}}
\def\GL{{\rm GL}}
\def\Q{{\mathbb Q}}

\def\mult{\mathrm{mult}}

\begin{document}

\title[Root systems in number fields]
{Root systems in number fields}

\author[Vladimir L. Popov]{Vladimir L. Popov${}^{1, 2}$}
\thanks{${}^1$  Steklov Mathematical Institute,
Russian Academy of Sciences, Gub\-kina 8, Moscow
119991, Russia, {\rm popovvl@mi-ras.ru}\\
\indent ${}^2$ National Research University
Higher School of Economics, Myas\-nitskaya
20, Moscow 101000,\;Russia}

\author[Yuri G. Zarhin]{Yuri G. Zarhin${}^{3}$}
\thanks{${}^{3}$ Department of Mathematics,
Pennsylvania State University, Uni\-versity Park, PA
16802, USA, {\rm zarhin@math.psu.edu} \\
\indent The second named author (Y.\,Z.) is partially
supported by Simons Foundation Collaboration grant \# 585711.
Part of this work was done during his stay in May--July 2018
at the Max-Planck-Institut f\"ur Mathematik (Bonn, Germany),
whose hospitality and support are gratefully acknowledged}

\begin{abstract}
We classify the types of root systems $R$ in the rings
of integers of number fields  $K$ such that the Weyl
group $W(R)$ lies in the group $\mathcal L(K)$ generated
by ${\rm Aut}  (K)$ and multipli\-ca\-tions by the elements of $K^*$. We also
classify the Weyl groups of root systems of rank $n$
which are isomorphic to a subgroup of $\mathcal L(K)$
for a number field $K$ of degree $n$ over $\mathbb Q$.
\end{abstract}

\maketitle

\section{Introduction}

  In what follows, we call the type of a (not necessarily reduced)
  root system the type of its Dynkin diagram.

   Let $L$ be a free Abelian group of a finite rank $n>0$.\;We shall
   consider it as a lattice of full rank in the $n$-dimensional
   linear space $V:=L\otimes_{\mathbb Z}{\mathbb Q}$ over $\mathbb Q$.
   Since every root is an integer linear combination of simple roots, for every type
    ${\sf R}$ of the root systems of rank  $n$, there is a
    subset $R$ in $L$ of rank $n$, which is a root system of
    type ${\sf R}$.\;However, if the pair $(V, L)$ is endowed with an
    additional structure, then the Weyl group $W(R)$ of such a
    realization may be inconsistent with this structure.\;Say, if
    the space $ V $ is endowed with a scalar product,
then it may happen that the group $W(R)$ does not preserve it
 (for instance, if $n=2$ and $e_1, e_2$ is an orthonormal
 basis in $L$, then
$\{\pm e_1, \pm e_2, \pm(e_1+e_2)\}$ is the root system of
type ${\sf A}_2$ in $V$,
whose Weyl group does not consist of orthogonal
transformations).\;Therefore, it is of interest only finding
such realizations, the Weyl group of which is consistent with
some additional structures on the pair $(V, L)$.

 A natural source of pairs $(V, L)$ is algebraic number theory,
 in which they arise in the form $(K, {\mathscr O}_K)$, where $K$ is a
 number field, and ${\mathscr O}_K$ is its ring of integers.\;In this case,
 three subgroups are naturally distingui\-shed in the group
${\rm GL}_ {\mathbb Q}(K)$ of nondegenerate linear
transfor\-ma\-tions of the linear space $K$ over $\mathbb Q$.\;The
first one is the
automorphism group ${\rm Aut} (K)$ of the field $K$.\;The
second is the image of the group monomor\-phism
\begin{equation}\label{mult}
{\rm mult}\colon K^* \hookrightarrow {\rm GL}_{\mathbb Q}(K),
\end{equation}
where
${\rm mult}(a)$ is the operator of multiplication by
 $a\in K^*$:
 \begin{equation}\label{mult1}
 {\rm mult}(a)\colon K\to K,\;x\mapsto ax.
 \end{equation}
 The third one is the subgroup ${\mathcal L}(K)$ in
 ${\rm GL}_{\mathbb Q}(K)$ generated by ${\rm Aut} (K)$
 and ${\rm mult}(K^*)$.

\begin{definition}\label{def}
We say that {\it the type ${\sf R}$ of {\rm(}\hskip -.3mm not necessarily
reduced\,{\rm)} root systems admits a realization in
the number field} $K$ if
\begin{enumerate}[\hskip 3.2mm\rm(a)]

\item $[K:\mathbb Q]={\rm rk} ({\sf R})$;

\item there is a subset $R$ of rank ${\rm rk} ({\sf R})$
in ${\mathscr O}_K$, which is
a root system of type ${\sf R}$;

\item $W(R)$ is a subgroup of the group ${\mathcal L}(K)$.
\end{enumerate}
\noindent In this case, the set $R$ is called {\it a realization
of the type ${\sf R}$ in the field\;$K$}.
\end{definition}

It is worth noting that if we replace $\mathscr O_K$ by $K$ in (b), we do not obtain a broader concept.\;Indeed, if $R$ is a subset of rank ${\rm rk}({\sf R})$ in $K$, which is a root system of type ${\sf R}$  such that (a) and (c) hold, then there is a positive integer $m$ such that $m\cdot R:=\{m\alpha\mid \alpha\in R\}\subset {\mathscr O}_K$.\;Clearly the set $m\cdot R$
has rank ${\rm rk}({\sf R})$, it
is a root system of type ${\sf R}$, and  $W(m\cdot R)=W(R)$.

In view of Definition \ref{def}, if a type $\sf R$ of root
systems admits a realization in a
number field $K$,
then the group $\mathcal L(K)$ contains a subgroup
isomor\-phic to the Weyl group of a root system of type $\sf R$.
Our first main result is the classification of all the cases
when the latter property holds:

\begin{theorem}\label{repres}
 The following properties of the Weyl group $W(R)$
 of a reduced root system $R$
of type ${\sf R}$ and rank $n$ are equivalent:
\begin{enumerate}[\hskip 2.2mm\rm(i)]
\item $W(R)$ is isomorphic to a subgroup $G$ of the group
${\mathcal L}(K)$, where $K$ is a number
field of degree $n$ over $\mathbb Q$;
\item $\sf R$ is contained in the following list:
\begin{gather}\label{Weyl}{\sf A}_1,\;\,{\sf A}_2,\;\,{\sf B}_2,\;\,{\sf G}_2,\;\,2{\sf A}_1,\;\, 2{\sf A}_1\stackrel{.}{+}{\sf A}_2,\;\,
{\sf A}_2\stackrel{.}{+}{\sf B}_2.
\end{gather}
\end{enumerate}
\end{theorem}

The fact that a subgroup $G$ of the group $\mathcal L(K)$
is isomorphic to the Weyl group of a root system of
rank $n=[K:\mathbb Q]$ and of type ${\sf R}$ is not
equivalent to the fact that  $G=W(R)$, where $R$ is a
root system of type ${\sf R}$ in ${\mathscr O}_K$.\;This is seen
from comparing Theorem \ref{repres} with
our second main result.\;The latter answers the question
of which of the types of root systems in list \eqref{Weyl}
are realized in number fields:

\begin{theorem}\label{R} For every type ${\sf R}$ of
root systems {\rm(}not necessarily reduced\,{\rm)}, the following properties are equivalent:
\begin{enumerate}[\hskip 2.0mm\rm(i)]

\item there is a number field, in which  ${\sf R}$ admits
a realization;

\item ${\rm rk}(\sf R)=1$ or $2$.
\end{enumerate}
\end{theorem}

For ${\rm rk}(\sf R)=1$ or $2$, the specific realizations
of $\sf R$ in number fields see in Section \ref{s1}.

\vskip 4mm

We are grateful to the referee for
informative
comments, which cont\-ri\-buted to making tangible improve\-ments of the text of this paper.

\vskip 4mm

\noindent{\it Terminology and notation}

\vskip 1mm

If $R$ is a root system of type
$\sf R$,
then the type
of the direct sum $mR$ of $m$ copies of $R$ is denoted by
$m{\sf R}$. We say that $\sf R$ is irreducible if $R$\;is.

All root systems of type ${\sf R}$ have the same rank
denoted by ${\rm rk}({\sf R})$.

${\sf A}'_1$ is the unique type of nonreduced root systems of rank $1$.

By a number field $K$ we mean a field that is an extension of a
finite degree of the field\;$\mathbb Q$.

$\mu_K$ is the multiplicative
group of all roots of unity in $K$; it is a finite cyclic group.

${\mathscr O}_K$ is the ring of all integers in $K$.

${\mathscr O}_K(d)$ is the set of all elements of ${\mathscr O}_K$, whose norm is $d$.

${\rm ord}(g)$ is the order of an element $g$ of a group

 $\langle g\rangle$ is a cyclic group with the generating element $g$.

 $\exp(G)$ is the exponent of a finite group $G$ (i.e., the least common multiple of the orders of all elements of $G$).

 $[G, G]$ is  the commutator subgroup of a group $G$.

 ${\rm Sym}_n$ and ${\rm Alt}_n$ are respectively the symmetric and alternating groups of permutations of $\{1,\ldots, n\}$.

For a prime number $p$ and a non-zero integer $n$,
the $p$-adic valuation of $n$ is denoted by $\nu_p(n)$
(i.e., $\nu_p(n)$ is the highest exponent $e$ such that $p^e$ divides $n$).

$\varphi$ is Euler's totient function, i.e., for every
integer $d>0$, the value $\varphi(d)$  is the number
of positive integers
relatively prime to\;$d$ and $\leqslant d$.

\section{Ranks $1$ and $2$}\label{s1}

The following examples show that
every type ${\sf R}$ of root systems of rank 1
or  2 admits a realization in an number field $K$.

\eject

{\it Root systems of types ${\sf A}_1$ and
${\sf A}'_1$}.

\vskip 1.5mm

In this case,  $K=\mathbb Q$, ${\mathscr O}_K=\mathbb Z$
and $\mathcal L(K)={\rm mult}(\mathbb Q^*)$.
If $\alpha\in \mathbb Z$, $\alpha\neq 0$,
then $R:=\{\pm\alpha\}$
(respectively,  $R:=\{\pm\alpha, \pm2\alpha\}$)
is a realization of type  ${\sf A}_1$ (respectively,
${{\sf  A}}'_1$) in the field $K$, because
$W(R)=\langle{\rm mult}(-1)\rangle\subset \mathcal L(K)$.

\vskip 2mm

{\it Root systems of types $\sf A_2$
and $\sf G_2$}.

\vskip 1.5mm

Let $K$ be the third cyclotomic field: $K=
\mathbb Q(\sqrt{-3})$.\;Then ${\mathscr O}_K=\mathbb Z+
\mathbb Z\omega$, where $\omega= (1+i\sqrt{3})/2$,
and ${\rm Aut} (K)=\langle c\rangle$, where $c$
is the complex conjugation $a\mapsto \overline a$.
The bilinear form
\begin{equation}\label{Eu}
K \times K\to \mathbb Q,\;
(a, b)\mapsto {\rm trace}_{K/\mathbb Q}(a\overline{b})=
2{\rm Re}(a\overline{b}),
\end{equation}
 defines on $K$
a structure of Euclidean space over $\mathbb Q$.\;Any
element of  $\mathcal L(K)$, whose order is finite
(in particular, any reflection), is an ortho\-go\-nal
(with respect to this structure) transformation.

Since $r_1:={\rm mult}(-1)c\in \mathcal L(K)$
is a reflection with respect to $1$,  the transformation
$\rho r_1{\rho}^{-1}$ for every
$\rho\in {\rm GL}_{\mathbb Q}(K)$ is
a reflection with respect to $\rho(1)$.
For $\rho={\rm mult}(a)$, where $a\in K^*$,
this yields the element
\begin{equation}\label{refl}
r_a:={\rm mult}(-a{\overline{a}}^{-1})c
\end{equation}
of $\mathcal L(K)$, which is a reflection with
respect to $a$.

The multiplicative group $\{\pm 1, \pm \omega,
\pm \omega^2\}$ of all $6$th roots of 1 coincides
with ${\mathscr O}_K(1)$.  Hence
\begin{equation*}
{\mathscr O}_K(1)=\{\pm\alpha_1,
\pm \alpha_2, \pm (\alpha_1+\alpha_2)\},\;
\mbox{where $\alpha_1=1$, $\alpha_2=\omega^2$.}
\end{equation*}
Therefore,
${\mathscr O}_K(1)$ is the root system of type ${\sf A}_2$
with the base $\alpha_1, \alpha_2$.\;If $a\in {\mathscr O}_K(1)$, then $-a\overline{a}^{-1}$ is a root of $1$, hence \eqref{refl} implies that
$r_a({\mathscr O}_K(d))={\mathscr O}_K(d)$ for any $d$.\;Therefore, $r_a\in
W({\mathscr O}_K(1))$.\;Hen\-ce $W({\mathscr O}_K(1))\subset
\mathcal L(K)$. This means that
${\mathscr O}_K(1)$ is the realization of type
${\sf A}_2$ in the field $K$.

Since we have
\begin{equation}\label{o3}
{\mathscr O}_K(3)=(1+\omega){\mathscr O}_K(1),
\end{equation}
the set
${\mathscr O}_K(3)$ is the root system of type ${\sf A}_2$
with the base
$$\beta_1=(1+\omega)\alpha_1, \beta_2=
(1+\omega)\alpha_2.$$
If $a\!\in\! {\mathscr O}_K(3)$, then \eqref{o3} implies that $-a\overline{a}^{-1}$ is a root of $1$.\;This and \eqref{refl} yield
$r_a({\mathscr O}_K(d))\!=\!{\mathscr O}_K(d)$ for any $d$.\;Therefore,
$W({\mathscr O}_K(3))\subset \mathcal L(K)$. Hence
${\mathscr O}_K(3)$ is yet another realization of type
${\sf A}_2$ in the field $K$.

\eject

\begin{figure}[h!]
\begin{center}\hskip -17mm
\includegraphics[width=.64\textwidth]{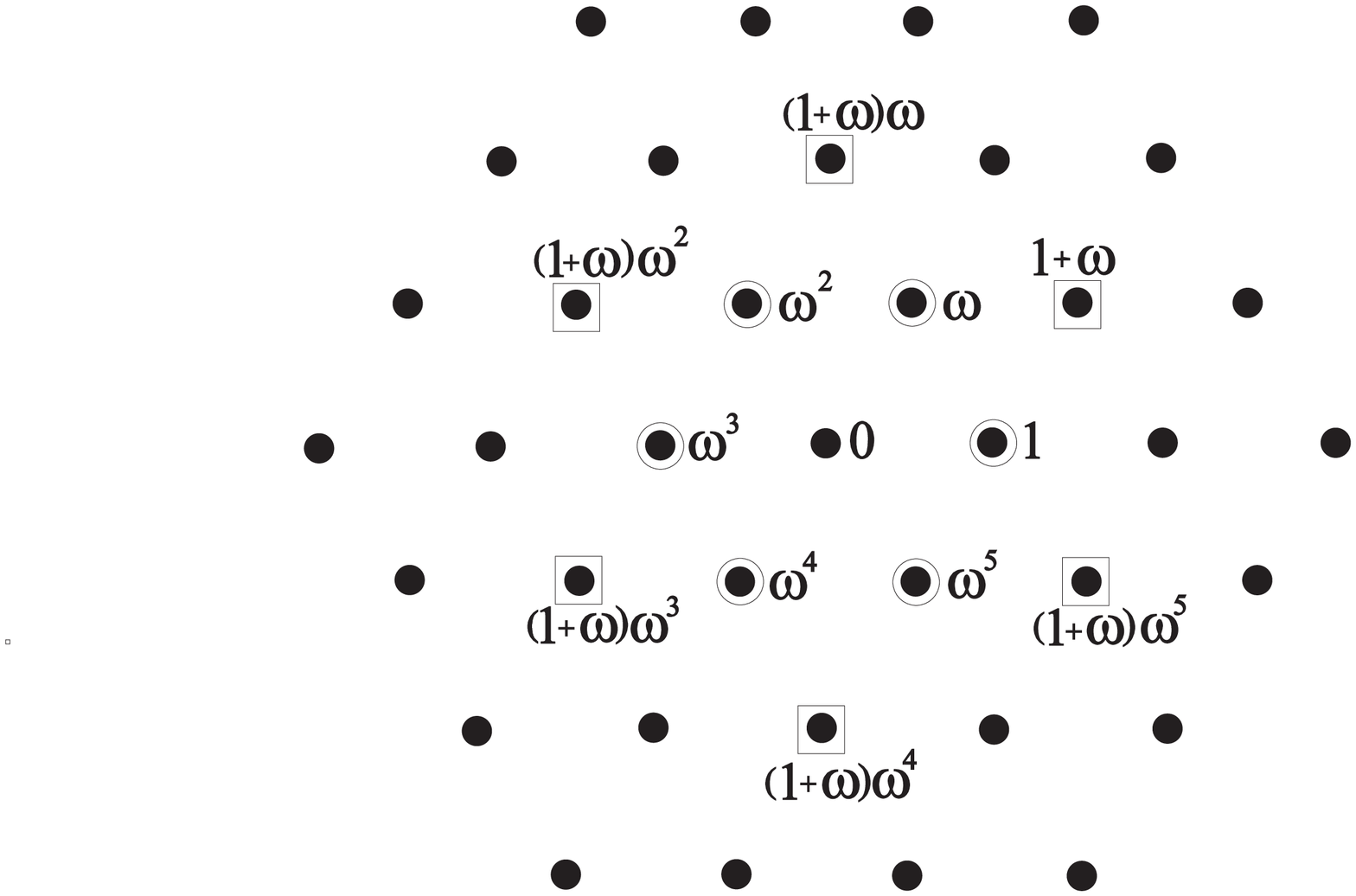}
\end{center}
\end{figure}

\

\vskip -15mm

\

\begin{quote}
{\small {\sc Firure 1.} Elements of ${\mathscr O}_K$, ${\mathscr O}_K(1)$,
and ${\mathscr O}_K(3)$ are depicted respectively by $\bullet$,
$\bullet$\hskip -2.5mm $\circledcirc$, and
$\bullet \hskip -2.5mm \raisebox{-.17\height}{$\boxdot$}$
}
\end{quote}

\vskip 2mm

Since we have
\begin{align*}
{\mathscr O}_K&(1)\textstyle\bigcup {\mathscr O}_K(3)\\
&=\{\pm\alpha_1,
\pm\beta_2, \pm (\alpha_1\!+\!\beta_2),
\pm (2\alpha_1\!+\!\beta_2), \pm (3\alpha_1\!+\!\beta_2),
\pm (3\alpha_1\!+\!2\beta_2)\},
\end{align*}
 the set ${\mathscr O}_K(1)
\bigcup {\mathscr O}_K(3)$ is the root system of type ${\sf G}_2$
with the base $\alpha_1, \beta_2$ (this is
noted in \cite[V, 16]{Se01}).
If $a\in {\mathscr O}_K(3)$, $b\in {\mathscr O}_K(1)$, then
$r_a({\mathscr O}_K(1))={\mathscr O}_K(1)$,
$r_b({\mathscr O}_K(3))={\mathscr O}_K(3)$.\;Therefore,
$W\big({\mathscr O}_K(1)\bigcup {\mathscr O}_K(3)\big)\subset
\mathcal L(K)$. Hence ${\mathscr O}_K(1)\bigcup {\mathscr O}_K(3)$
is the realization of type ${\sf G}_2$ in the field\;$K$.

\vskip 2mm

{\it Root systems ${\sf B}_2$,  $2{\sf A}_1
$, ${\sf BC}_2$,  $2{\sf A}_1$, $2{\sf A}'_1$, and ${\sf A}_1\stackrel{.}+{\sf A}'_1$}.

\vskip 1.5mm

Let $K$ be the fourth cyclotomic field: $K=
\mathbb Q(\sqrt{-1})$. Then ${\mathscr O}_K=\mathbb Z+
\mathbb Z i$ and ${\rm Aut} (K)=\langle c\rangle$,
where $c$ is the complex conjugation $a\mapsto \overline a$.
As above, \eqref{Eu} defines on $K$ a structure of
Euclidean space over $\mathbb Q$, and  any
element of $\mathcal L(K)$ of finite order
(in particular, any reflection) is an orthogonal
(with respect to this structure) transformation.

As above, for every  $a\in K^*$, the element
$r_a\in\mathcal L(K)$, given by formula \eqref{refl},
is a reflection with respect to $a$.

The multiplicative group $\{\pm 1, \pm i\}$
of all $4$th roots of 1 coincides with  ${\mathscr O}_K(1)$.\;Therefore,
\begin{equation*}
{\mathscr O}_K(1)=\{\pm\alpha_1, \pm\alpha_2\}
\end{equation*}
is the root system of type
$2{\sf A}_1
$ with the
base $\alpha_1\!=\!1, \alpha_2\!=\!i$.\;If $a\!\in\!
{\mathscr O}_K(1)$, then  $-a\overline{a}^{-1}$ is a root of $1$, hence \eqref{refl} implies that
$r_a({\mathscr O}_K(d))={\mathscr O}_K(d)$ for any $d$.
Therefore,
$r_a\!\in\! W({\mathscr O}_K(1))$; whence $W({\mathscr O}_K(1))\!\subset\!
\mathcal L(K)$.\;So, ${\mathscr O}_K(1)$ is the
realization of type $2{\sf A}_1
$ in\;$K$.

Since we have
\begin{equation}\label{o22}
{\mathscr O}_K(2)=(1+i){\mathscr O}_K(1),
\end{equation}
the set
${\mathscr O}_K(2)$ is the root system of type $2{\sf A}_1
$ with the base
$$\beta_1=(1+i)\alpha_1,\; \beta_2=(1+i)
\alpha_2.$$
If $a\!\in\! {\mathscr O}_K(2)$,
then \eqref{o22} implies that $-a\overline{a}^{-1}$ is a root of $1$.\;This and \eqref{refl} yield
$r_a({\mathscr O}_K(d))\!=\!{\mathscr O}_K(d)$ for any $d$.\;Therefore,
$W\big({\mathscr O}_K(2)\big)\!\subset\! \mathcal L(K)$.
Hence ${\mathscr O}_K(2)$ is yet another realization of
type $2{\sf A}_1
$ in $K$.

Since we have
$${\mathscr O}_K(1)\textstyle \bigcup {\mathscr O}_K(2)=\{\pm \alpha_1,
\pm\beta_2, \pm(\alpha_1+\beta_2),
\pm(2\alpha_1+\beta_2)\},$$
the set ${\mathscr O}_K(1)
\bigcup {\mathscr O}_K(2)$ is the root system of type ${\sf B}_2$
with the base $\alpha_1, \beta_2$. If $a\in
{\mathscr O}_K(2)$, $b\in {\mathscr O}_K(1)$, then $r_a({\mathscr O}_K(1))={\mathscr O}_K(1)$,
$r_b({\mathscr O}_K(2))={\mathscr O}_K(2)$. Therefore,
$W\big({\mathscr O}_K(1)\bigcup {\mathscr O}_K(2)\big)\subset
\mathcal L(K)$, hence ${\mathscr O}_K(1)\bigcup {\mathscr O}_K(2)$
is the realization of type ${\sf B}_2$ in the field $K$.

Since we have
$$
{\mathscr O}_K(4)=2{\mathscr O}_K(1),
$$
the group $W({\mathscr O}_K(4))$ coincides with
the group $W({\mathscr O}_K(1))$.\;Therefore,  ${\mathscr O}_K(4)$
is yet another realization
of type $2{\sf A}_1
$ in $K$. Since
\begin{align*}
\textstyle {\mathscr O}_K(1)\bigcup &{\mathscr O}_K(2)\textstyle\bigcup {\mathscr O}_K(4)\\
&=
\{\pm \alpha_1, \pm 2\alpha_1,  \pm\beta_2,
\pm(\alpha_1+\beta_2), \pm 2(\alpha_1+\beta_2),
\pm(2\alpha_1+\beta_2)\},
\end{align*}
the set
${\mathscr O}_K(1)\bigcup {\mathscr O}_K(2)\bigcup {\mathscr O}_K(4)$
is the root system of type ${\sf BC}_2$ with the base
$\alpha_1, \beta_2$. In view of
$W\big({\mathscr O}_K(1)\bigcup {\mathscr O}_K(2)\bigcup {\mathscr O}_K(4)\big)
\subset \mathcal L(K)$, it is the realization of
type ${\sf BC}_2$ in $K$.

\begin{figure}[h!]
\begin{center}\hskip -0mm
\includegraphics[width=.481\textwidth]{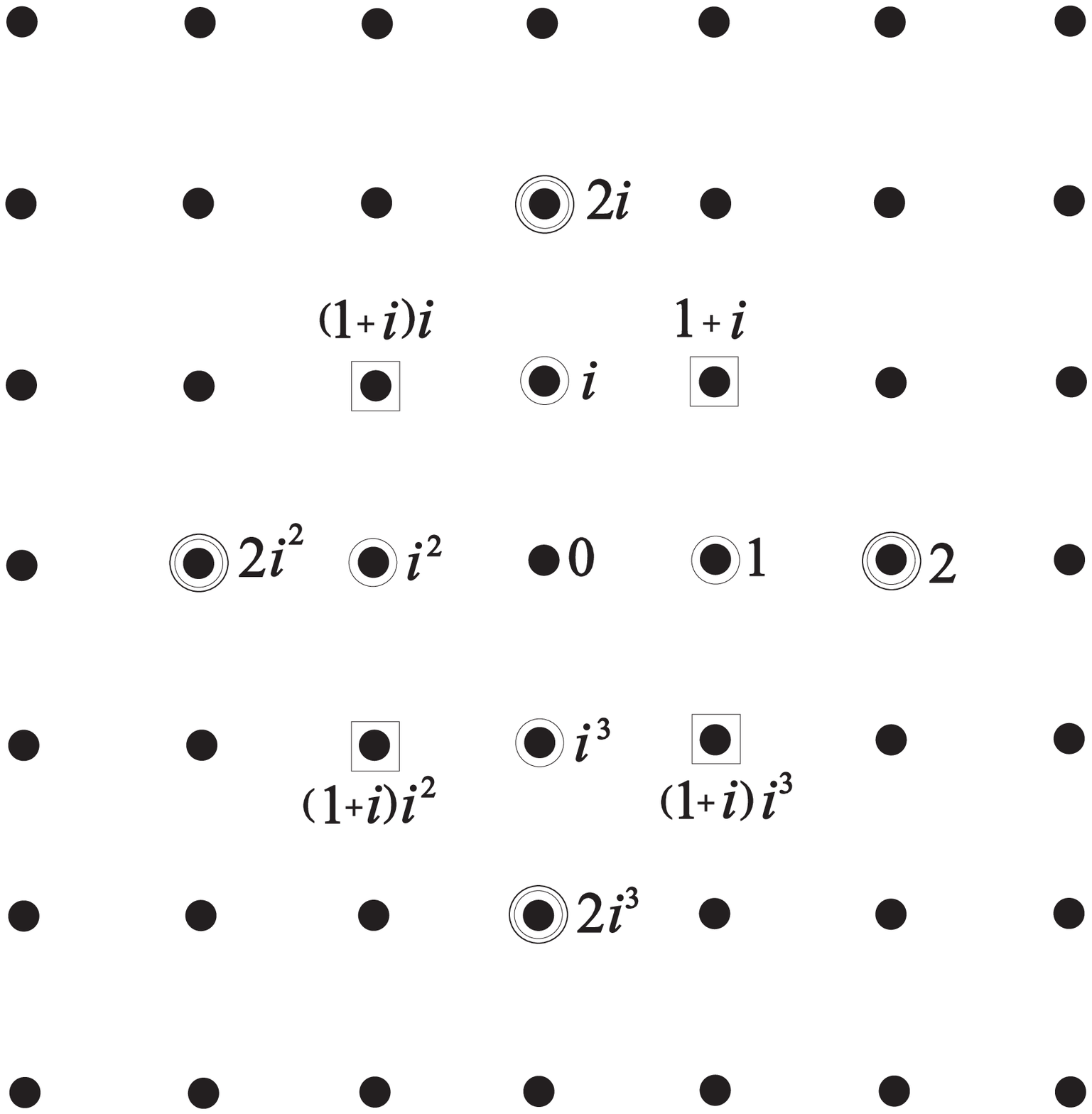}
\end{center}
\end{figure}

\

\vskip -15mm

\

\begin{quote}
{\small {\sc Firure 2.} Elements of ${\mathscr O}_K$, ${\mathscr O}_K(1)$,
${\mathscr O}_K(2)$, and ${\mathscr O}_K(4)$ are depicted respectively by
$\bullet$, $\bullet$\hskip -2.5mm $\circledcirc$,
$\bullet \hskip -2.5mm \raisebox{-.17\height}{$\boxdot$}$,
and \hskip 1mm $\bullet$\hskip -2.6mm
\raisebox{-.02\height}{$\circledcirc$}  \hskip -4.68mm
\raisebox{-.10\height}{$\mbox{\large $\circledcirc$}$}
}
\end{quote}

\vskip 2mm

Finally,  the realizations of types  $2{\sf A}'_1$ and
${\sf A}_1\stackrel{.}+{\sf A}'_1$
in $K$ are respecti\-ve\-ly
$$
{\mathscr O}_K(1)\textstyle \bigcup {\mathscr O}_K(4)\;\; \mbox{and}\;\; {\mathscr O}_K(1)
\bigcup \{\pm 2\}.
$$

Summing up, we have the following
\begin{proposition}\label{rank12}
Every type of root systems of rank $\leqslant 2$ admits
a re\-a\-li\-zation in a number field.
\end{proposition}

 \section{Group $\mathcal L(K)$ and its finite subgroups}

Below $K$ is a number field of degree $n$ over $\mathbb Q$. Given $g\in {\mathcal L}(K)$, we denote
\begin{equation*}
K^g:=\{a\in K\mid g(a)=a\}.
\end{equation*}
If $g\in  {\rm Aut}(K)$, then $K^g$ is a subfield of $K$ and $K/K^g$ is a degree
${\rm ord}(g)$ Galois extension with the Galois group $\langle g\rangle$.

\begin{lemma}
\label{keyK} \
\begin{asparaitem}\itemsep 1ex
\item[\rm(a)] The group ${\mathcal L}(K)$ is a semidirect product of
its normal subgroup ${\rm mult}(K^*)$ and the subgroup
${\rm Aut} (K)$.\;Therewith,
\begin{equation}\label{conj}
g\,{\rm mult}(a)\,g^{-1}={\rm mult}(g(a)\!)\quad
\mbox{for any $a\in K^*$, $g\in {\rm Aut} (K)$.}
\end{equation}
\item[\rm(b)]
Let $g\in  {\rm Aut}(K)$ and let
$${\rm Norm}_{K/K^{g}}\colon K\to K^g$$ be the norm map corresponding to the field extension $K/K^g$.\;Then for any $a\in K^*$,
 the following properties are equivalent:
    \begin{enumerate}[\hskip 9.9mm \rm(i)]
    \item[$({\rm b}_1)$] the element ${\rm mult}(a)g\in {\mathcal L}(K)$ has finite order;
    \item[$({\rm b}_2)$]
        ${\rm Norm}_{K/K^{g}}(a)$ is a root of unity in $K^g$.
    \end{enumerate}
\end{asparaitem}
\end{lemma}

\begin{proof}
(a) First, check that the set of all products $\mult(a)g$,
where $a\in K^*$, $g\in {\rm Aut}(K)$, is a subgroup of
$\GL_{\Q}(K)$.\;Let $a_1, a_2\in K^{*}$ and $g_1, g_2\in
\Aut(K)$.\;Then, for each $v \in K$,
$$\mult(a_1)g_1\mult(a_2)g_2 (v)=  a_1 (g_1 (a_2 g_2 (v)\!)\!)=
a_1 (g_1 (a_2)\!) (g_1 (g_2 (v)\!)\!).$$
This yields
\begin{equation}
\label{productK}
\mult(a_1)g_1\mult(a_2)g_2=\mult(a_1 (g_1 (a_2)\!)\!) g_1 g_2.
\end{equation}
From \eqref{productK} we infer  that the inverse of
$\mult(a)g$ is $\mult(g^{-1}(a^{-1})\!)g^{-1}$. Thus
the set of all products $\mult(a)g$ is a subgroup of $\GL_{\Q}(K)$.

On the other hand,
$$\mult(a)g (1)
={\rm mult}(a) (g(1))={\rm mult}(a) (1)
=a,$$

\noindent hence the linear operator $\mult(a)g$ uniquely
determines $a$, and there\-fo\-re, $g$ as well.  This
implies that the map
\begin{equation}\label{psi}
\psi\colon \mathcal{L}(K) \to \Aut(K), \quad \mult(a)g \mapsto g,
\end{equation}
is well defined.\;By
\eqref{productK}, the map \eqref{psi} is a
group epimorphism and
\begin{equation}\label{kerpsi}
\ker(\psi)=\mult(K^{*}).
\end{equation}
Finally, \eqref{conj}  straightforwardly follows
from \eqref{productK}.

(b) Let $m:={\rm ord}(g)$.\;Then it follows from \eqref{productK} that
\begin{equation}\label{fo}
({\rm mult}(a)g)^m=\textstyle{\rm mult}\Big(\prod_{d=0}^{m-1}g^d(a)\Big)g^m={\rm mult} \Big(\prod_{d=0}^{m-1}g^d(a)\Big).
\end{equation}
In turn, \eqref{fo} implies that $({\rm b}_1)$ holds if and only if
the right-hand side of \eqref{fo} is a root of unity.\;But this right-hand side is
${\rm Norm}_{K/K^{g}}(a)$.
\end{proof}

\begin{remark} {\rm The examples below show that, given $a\in K^*$, $g\in {\rm Aut}(K)$,
the element ${\rm mult}(a)  g\in {\mathcal L}(K)$ may have a finite order while $a$ is not a root of unity (i.e., its order is infinite).
\begin{exsnonumber}
1. Let $K=\mathbb Q(\sqrt{2})$.\;Then $|{\rm Aut}(K)|=2$. Let $g\in {\rm Aut}(K)$, $g\neq {\rm id}_K$, and $a=1+\sqrt{2}$.\;Then ${\rm ord}({\rm mult}(a)g)=4$, ${\rm ord}(a)=\infty$.

2. Let  $K={\mathbb Q}(\sqrt{-1})$.\;Then $|{\rm Aut}(K)|=2$. Let $g\in {\rm Aut}(K)$, $g\neq {\rm id}_K$, and $a=(3+4\sqrt{-1})/5$.\;Then ${\rm ord}({\rm mult}(a)g)=2$, ${\rm ord}(a)=\infty$.
\end{exsnonumber}
}
\end{remark}

\begin{lemma}
\label{key}
For any finite subgroup $G$ of $\mathcal L(K)$,
there is a {\rm(}cyclic\,{\rm)} subgroup $H$ of $\mu_K$ such that ${\rm mult}(H)\subseteq G$ and
\begin{enumerate}[\hskip 3.2mm\rm(i)]
\item  the sequence
$1\to H\xrightarrow{{\rm mult}} G\xrightarrow{\psi}\psi(G)\to 1$
is exact;
\item $|G|=|H|\!\cdot\!|\psi(G)|$;
\item $|H|$ divides $|\mu_K|$;
\item $\varphi(|H|)$ divides $n$;
\item $|\psi(G)|$ divides $|\Aut(K)|$, which divides $n$;
\item if $p\geqslant 2$ is a prime integer, then
$\nu_p(|G|)\leqslant 2\nu_p(n)+1$;
\item $|G|$ divides $n\exp(G)$.
\end{enumerate}
\end{lemma}

\begin{proof} Since $\mu_K$ is the set of all elements
of finite order in $K^*$, and \eqref{mult} is a group
monomorphism, the existence of $H$ and (i) follow from \eqref{kerpsi}.
Since $\mu_K$ is cyclic, $H$ is cyclic as well.

Statements (ii), (iii), (v) are clear.

Let $\theta$ be a generator of the cyclic group $H$.\;Then
$[\mathbb Q(\theta):\mathbb Q]=\varphi({\rm ord}(\theta))=
\varphi(|H|)$. Whence (iv), because $\mathbb Q(\theta)$
is a subfield of\;$K$.

Since, in the above notation, ${\rm ord}(\theta)=|H|$, the definition of
$\exp(G)$ implies that $|H|$ divides $\exp(G)$. In view of (ii), this and
(v) yield\;(vii).

Let $\nu_p(|\mu_K|)=d$ and let $\zeta\in \mu_K$ be a
primitive $p^d$th root of unity. From
$\mathbb Q(\zeta)\subseteq K$ and $[\mathbb Q(\zeta):
\mathbb Q]=\varphi(p^d)=p^{d-1}(p-1)$ we infer that
$p^{d-1}(p-1)$ divides $n$. Hence $\nu_p(n)\geqslant d-1=
\nu_p(|\mu_K|)-1$.\;This and (ii), (iii), (v)  then imply $\nu_p(|G|)\leqslant \nu_p(|\mu_K|)+\nu_p(n)\leqslant 2\nu_p(n)+1$, which proves (vi).
\end{proof}

 For further proofs, Table 1 below summarizes some information  from \cite{Bo68} about the Weyl groups of irre\-ducible root systems $R$ of type $\sf R$:

\eject

{\fontsize{10pt}{5mm}\selectfont
\begin{center}{\fontsize{11pt}{6mm}\selectfont{\sc Table 1}}
\\[3mm]
\begin{tabular}{c||c|c|c|c}

$\sf R$
&
$\sf A_\ell$, {\fontsize{10pt}{6mm}\selectfont $\ell\!\geqslant \!1$}
&
$\sf B_\ell$, {\fontsize{10pt}{6mm}\selectfont $\ell\!\geqslant\! 2$}
&
$\sf C_\ell$, {\fontsize{10pt}{6mm}\selectfont $\ell\!\geqslant\! 2$}
&
$\sf D_\ell$, {\fontsize{10pt}{6mm}\selectfont $\ell\!\geqslant\! 4$}
\\
\hline
&&&&\\[-13pt]
$|W(R)|$
&
$(\ell+1)!$
&
$2^\ell\!\cdot\!\ell!$
&
$2^\ell\!\cdot\!\ell!$
&
$2^{\ell-1}\!\cdot\!\ell!$
\\
\end{tabular}
\end{center}
}

\

\vskip -10mm

\

{
\fontsize{10pt}{5mm}\selectfont
\begin{center}{\ }

\begin{tabular}{c||c|c|c|c|c}

$\sf R$
&
$\sf E_6$
&
$\sf E_7$
&
$\sf E_8$
&
$\sf F_4$
&
$\sf G_2$
\\
\hline
&&&&&\\[-13pt]
$|W(R)|$
&
$2^7\!\cdot\!3^4\!\cdot\!5$
&
$2^{10}\!\cdot\!3^4\!\cdot\!5\!\cdot\!7$
&
$2^{14}\!\cdot\!3^5\!\cdot\!5^2\!\cdot\!7$
&
$2^7\!\cdot\!3^2$
&
$2^2\!\cdot\!3$
\\
\end{tabular}
\end{center}
}

\

\vskip -10mm

\

{
\fontsize{10pt}{5mm}\selectfont
\begin{center}{\ }

\begin{tabular}{c||c|c|c|c}
$\sf R$ & ${\sf A}_1$ &${\sf A}_2$ & ${\sf B}_2={\sf C}_2$ &  ${\sf G}_2$
\\
\hline
&&&&\\[-13pt]
$\exp(W(R))$ & 2 & 6 & 4 & 6
\\
\end{tabular}
\end{center}
}

\vskip 4mm

\begin{lemma}\label{W>}
Let $W(R)$ be the Weyl group of a root system $R$.\;Then
\begin{equation}\label{nu2}
\nu_2(|W(R)|)\geqslant \left [
{({\rm rk}(R)
+1)}/{2}\right].
\end{equation}
\end{lemma}

\begin{proof} First, by Legendre's formula (see, e.g., \cite[Thm.\;2.6.1\;]{Mo12}),
we have
$\nu_2(m!)= \sum_{k=1}^{\infty}\left[m/2^k\right]\geqslant [m/2]+[m/4]$, which readily yields
\begin{equation}\label{m!}
\nu_2(m!)\geqslant \left[(m+1)/2\right]\quad\mbox{for $m\neq 1, 3$.}
\end{equation}

Next, suppose that $R$ is irreducible of type $\sf R$.\;Then
\eqref{nu2} directly follows from
Table 1 and \eqref{m!}.

In the case of an arbitrary root system
\begin{equation}\label{sr}
R=R_1\stackrel{.}{+}\cdots\stackrel{.}{+}R_{d},
\end{equation}
where $R_i$ is an irreducible root system for every $i$,
we have
\begin{equation}\label{sura}
{\rm rk}(R)={\rm rk}(R_1)+\cdots + {\rm rk}(R_d)
\end{equation}
and $W(R)$ splits into the product
\begin{equation}\label{sW}
W(R)=W(R_1)\times\cdots\times W(R_d)
\end{equation}
where $W(R_i)$ is the Weyl group of
$R_i$.\;It follows from \eqref{sW} that, for every
prime integer $p\geqslant 2$,
\begin{equation}\label{nuW}
\nu_p(|W(R)|)=\sum_{i=1}^d  \nu_p(|W(R_i)|).
\end{equation}

Given
that for every $W(R_i)$ the desired inequa\-lity is proved,
we then deduce from \eqref{nuW} that
\begin{equation}\label{Wi}
\nu_2(|W(R)|)
\geqslant  \sum_{i=1}^d  \left[({\rm rk}(R_i)+1)/{2}\right].
\end{equation}

Now \eqref{nu2} follows from \eqref{Wi} because of the inequality
\begin{equation}
\label{inab}
 \left[(a+b+1)/{2}\right] \leqslant   \left[(a+1)/{2}\right]+
 \left[(b+1)/{2}\right]\!,
\end{equation}
which holds for all integers $a$ and $b$.\;To prove \eqref{inab},
note that if we replace $a$ by $a+2$, then  the both
sides of  \eqref{inab} would increase by $1$.\;So
validity of \eqref{inab} for any $a$ and $b$ follows from its evident validity
for $a=0$ and $a=1$.
\end{proof}

Below some of the arguments are based on the information
that readily follows from Table 1.\;It is convenient to
collect it in Table\;2, where we use the same notation
as in Table 1 and,  for every prime integer $p\geqslant 2$, put
$\nu_p({\sf R}):=\nu_p(|W(R)|)$.

{\fontsize{10pt}{5mm}\selectfont
\begin{center}{\fontsize{11pt}{6mm}\selectfont{\sc Table 2}}
\\[1mm]
${\sf R}={\sf A}_{\ell}$
\\[1mm]
\begin{tabular}{c||c|c|c|c|c|c|c|c|c|c|c|c|c|c|c|c}
$\ell$
&
$1$
&
$2$
&
$3$
&
$4$
&
$5$
&
$6$
&
$7$
&
$8$
&
$9$
&
${10}$
&
${11}$
&
${12}$
&
${13}$
&
${14}$
&
${15}$
&
${16}$
\\
\hline
&&&&\\[-13pt]
$\nu_2({\sf R})$
&
1
&
1
&
3
&
3
&
4
&
4
&
7
&
7
&
8
&
8
&
10
&
10
&
11
&
11
&
15
&
15
\\
\hline
&&&&\\[-13pt]
$\nu_3({\sf R})$
&
0
&
1
&
1
&
1
&
2
&
2
&
2
&
4
&
4
&
4
&
5
&
5
&
5
&
6
&
6
&
6
\\
\end{tabular}
\end{center}
}

\

\vskip -9mm

\

{\fontsize{10pt}{5mm}\selectfont
\begin{center}{\ }
\\[-2mm]
${\sf R}={\sf B}_{\ell}\;\;\mbox{\rm and}\;\;{\sf C}_{\ell} $
\\[1mm]
\begin{tabular}{c||c|c|c|c|c|c|c|c|c|c|c|c|c|c|c}
$\ell$
&
$2$
&
$3$
&
$4$
&
$5$
&
$6$
&
$7$
&
$8$
&
$9$
&
${10}$
&
${11}$
&
${12}$
&
${13}$
&
${14}$
&
${15}$
&
${16}$
\\
\hline
&&&&\\[-13pt]
$\nu_2({\sf R})$
&
3
&
4
&
7
&
8
&
10
&
11
&
15
&
16
&
18
&
19
&
22
&
23
&
25
&
26
&
31
\\
\hline
&&&&\\[-13pt]
$\nu_3({\sf R})$
&
0
&
1
&
1
&
1
&
2
&
2
&
2
&
4
&
4
&
4
&
5
&
5
&
5
&
6
&
6
\\
\end{tabular}
\end{center}
}

\

\vskip -9mm

\

{\fontsize{10pt}{5mm}\selectfont
\begin{center}{\ }
\\[-2mm]
${\sf R}={\sf D}_{\ell}$
\\[1mm]
\begin{tabular}{c||c|c|c|c|c|c|c|c|c|c|c|c|c}
$\ell$
&
$4$
&
$5$
&
$6$
&
$7$
&
$8$
&
$9$
&
$10$
&
$11$
&
${12}$
&
${13}$
&
${14}$
&
${15}$
&
${16}$
\\
\hline
&&&&\\[-13pt]
$\nu_2({\sf R})$
&
6
&
7
&
9
&
10
&
14
&
15
&
17
&
18
&
21
&
22
&
24
&
25
&
30
\\
\hline
&&&&\\[-13pt]
$\nu_3({\sf R})$
&
1
&
1
&
2
&
2
&
2
&
4
&
4
&
4
&
5
&
5
&
5
&
6
&
6
\\
\end{tabular}
\end{center}
}

\

\vskip -11mm

\

{\fontsize{10pt}{5mm}\selectfont
\begin{center}{\ }
\\[3mm]
\begin{tabular}{c||c|c|c|c|c}
$\sf R$
&
${\sf E}_6$
&
${\sf E}_7$
&
${\sf E}_8$
&
${\sf F}_4$
&
${\sf G}_2$
\\
\hline
&&&&\\[-13pt]
$\nu_2({\sf R})$
&
7
&
10
&
14
&
7
&
2 
\\
\hline
&&&&\\[-13pt]
$\nu_3({\sf R})$
&
4
&
4
&
5
&
2
&
1
\\
\end{tabular}
\end{center}
}

\vskip 1mm

Below, for every type ${\sf R}$ of root systems,
we put $\varnothing \stackrel{.}+{\sf R}:={\sf R}$,
$0{\sf R}=\varnothing$,
and,
by definition, ${\rm rk}(\varnothing)=
 \nu_p(\varnothing)=0$ for any $p$.

\begin{proposition}\label{l3}  Let $R$ be a reduced root system
of type ${\sf R}$.

\begin{enumerate}[\hskip 2.7mm\rm(i)]
\item If $\;{\sf R}={\sf S}_1\stackrel{.}+{\sf S}_2$,
then
$$\nu_2({\sf S}_1) \leqslant \nu_2({\sf R})-
[({\rm rk}({\sf S}_2)+1)/2].
$$
In particular, if $\;{\sf R}_i$ is the type of $R_i$ in
{\rm \eqref{sr}}, then
\begin{equation*}
\nu_2({\sf R}_i)  \leqslant  \nu_2({\sf R}) -
[(n-{\rm rk}({\sf R}_i)+1)/2]  < \nu_2({\sf R}).
\end{equation*}

\item  If  $\nu_2({\sf R}) \leqslant 3$, then
\begin{gather*}
 {\sf R}=a_1{\sf A}_1\stackrel{.}+a_2{\sf A}_2
 \stackrel{.}+a_3{\sf A}_3\stackrel{.}+a_4{\sf A}_4
 \stackrel{.}+b_2{\sf B}_2
 \stackrel{.}+g_2{\sf G}_2,\\
 a_1+2a_2+3a_3+4a_4+2b_2
 +2g_2={\rm rk}({\sf R}),\\
 a_1+a_2+3a_3+3a_4+3b_2
 +2g_2=\nu_2({\sf R}),\\
 a_2+a_3+a_4
 +g_2=\nu_3({\sf R}).
 \end{gather*}
\item
 If  $\nu_3({\sf R}) \leqslant 1$, then
 \begin{gather}
 {\sf R}= {\sf X}\stackrel{.}+a_1{\sf A}_1\stackrel{.}+
 b_2{\sf B}_2,\;\;\mbox{\rm where}\label{nu31}\\
  {\sf X}\in \{{\sf A}_2, {\sf A}_3, {\sf A}_4,  {\sf B}_3,
  {\sf B}_4,   {\sf B}_5, {\sf C}_3, {\sf C}_ 4, {\sf C}_ 5,
  {\sf D}_4, {\sf D}_5, {\sf G}_2, \varnothing\},\label{nu32}\\
 {\rm rk}({\sf X})+a_1+2b_2={\rm rk}({\sf R}),\notag\\
 \nu_2({\sf X})+a_1+3b_2=\nu_2({\sf R}).\notag
 \end{gather}
 In this case, $W(R)$ contains a subgroup isomorphic to
 the Weyl group of a root system
of type $m{\sf A}_1$,
where
 \begin{equation*}
m= \begin{cases}
{\rm rk}({\sf R})&\hskip -2mm\mbox{if $\;{\sf X}={\sf B}_3,
  {\sf B}_4,   {\sf B}_5, {\sf C}_3, {\sf C}_ 4, {\sf C}_ 5,
  {\sf D}_4, {\sf G}_2, \varnothing
$},\\
{\rm rk}({\sf R})-1 &  \hskip -2mm\mbox{if $\;{\sf X}={\sf A}_2, {\sf A}_3, {\sf D}_5$},\\
{\rm rk}({\sf R})-2 &\hskip -2mm\mbox{if $\;{\sf X}= {\sf A}_4$}.
 \end{cases}
\end{equation*}
 \end{enumerate}
\end{proposition}

\begin{proof}
Combining  Lemma \ref{W>}, \eqref{sura}, \eqref{nuW}, and Table\;2 yields the proof of
all but the last claim of (iii).\;To prove the latter, we note that if $Y$ is a root system of type ${\sf Y}$, then, for ${\sf Y}={\sf A}_1$, ${\sf B}_{\ell}$, ${\sf C}_{\ell}$, ${\sf D}_{\ell}$ ($\ell$ even), and ${\sf G}_2$, there is a subset
of ${\rm rk}({\sf Y})$ strongly orthogonal (hence orthogonal) roots in $Y$ (see \cite[Thms.\;3.1, 5.1]{AK02}; cf.\;\cite[Chap.\;VI, \S\;1, Sect. 3, Cor. of Thm. 1, Exer.\;15]{Bo68}).\;The subgroup of $W(Y)$ generated by reflections corresponding
to these roots is isomorphic to the Weyl group of a root system of type ${\rm rk}({\sf Y}){\sf A}_1$.\;This proves the claim in the case ${\sf X}={\sf B}_3$,
  ${\sf B}_4$,   ${\sf B}_5$, ${\sf C}_3$, ${\sf C}_ 4$, ${\sf C}_ 5$,
  ${\sf D}_4$, ${\sf G}_2$, $\varnothing$.\;The root systems of types ${\sf A}_2$, ${\sf A}_3$, ${\sf A}_4$, and ${\sf D}_5$ clearly contain closed subsystems of  respectively types ${\sf A}_1$,
  ${\sf A}_1\stackrel{.}+{\sf A}_1$, ${\sf A}_1\stackrel{.}+{\sf A}_1$, and ${\sf D}_4$.\;This implies the claim in the remaining two cases.
\end{proof}

\begin{proposition}\label{mA1} Let $K$ be a number field of degree $n$ over $\mathbb Q$ and let $m$ be a positive integer.\;If
the group $\mathcal L(K)$ contains a subgroup $G$ isomorphic to the Weyl group
of a root system
of type $m{\sf A}_1$, then
the following hold:
\begin{enumerate}[\hskip 4.2mm \rm(a)]
\item $2^{m-1}$ divides $n$;
\item if $m=n$, then
$n=1$ or $2$;
\item if $m=n-1$, then
$n=2$ or $4$;
\item if $m=n-2$, then
$n=3$ or $4$.
\end{enumerate}
\end{proposition}
\begin{proof}
Table 1 and \eqref{sW} yield $\exp(G)=2$ and $|G|=2^m$.\;Hence, by Lem\-ma \ref{key}(vii),
$2^m$ divides $2n$, whence (a).\;In turn, (a) implies  (b), (c), and (d).
\end{proof}

\begin{proposition}\label{124} Let $K$ be a
number field of degree $n$ over $\mathbb Q$.\;If the
group $\mathcal L(K)$ contains a finite subgroup $G$
isomorphic to the Weyl group $W(R)$ of a root
system $R$ of type ${\sf R}$ and rank $n$, then $n \in\{1, 2, 4\}$.
\end{proposition}

\begin{proof} First, in Step 1, we shall show
that $n\in \{1, 2, 4, 6, 8, 16\}$.\;Then,  in Steps 2, 3,
and 4, we shall consider  respectively the cases $n=6, 8$,
and 16, and eliminate each of them.

\vskip 1mm

\indent {\it Step $1$.}

Lemma \ref{key}(vi) and \eqref{nu2} yield
\begin{equation}
\label{ineq}
 \left[(n+1)/{2}\right] \leqslant \nu_2(|G|)\leqslant  2 \nu_2(n)+1.
\end{equation}

Let $n\geqslant  3$.\;Then $2\leqslant
\left[(n+1)/{2}\right]$.\;In view of \eqref{ineq},
this implies that
$\nu_2(n)\geqslant  1$, i.e., $n$ is even. Since
$n/2\leqslant  \left[(n+1)/2\right]$, from \eqref{ineq} we infer
\begin{equation}\label{inequ}
2^{n/2}\leqslant  2^{\left[(n+1)/2\right]}\leqslant
2^{2\nu_2(n)+1}\leqslant  2 n^2.
\end{equation}
In addition, if $n$ is not a power of $2$, then
$2^{\nu_2(n)}\cdot 3\leqslant  n$, so \eqref{inequ} yields
\begin{equation}
\label{n6}
2^{n/2}\leqslant  2^{2\nu_2(n)+1}\leqslant  2 \left({n}/{3}\right)^2.
\end{equation}

If in \eqref{n6} we replace $n$ by $n+2$ then
the left-hand side will be multiplied by $2$
while the right-hand side will be multiplied by $(1+2/n)^2<2$,
because $n>4$.\;Taking into account
that  \eqref{n6} becomes equality if $n=6$, we conclude that
 $n=6$ if $n\geqslant  3$ is not a power of $2$.

Now suppose that $n=2^s$, where $s \geqslant  2$.
Then \eqref{ineq} yields
$$2^{s-1} \leqslant  2s+1$$
and therefore $s=2,3$ or $4$, i.e., $n=4,8$ or $16$ respectively.

Taking into account all $n<3$,  we conclude that $n\in\{1,2,4,6,8, 16\}$.

\vskip 2mm

In  Steps 2, 3, and 4, we use the notation of \eqref{sr},
\eqref{sW} introduced in the proof of Lemma \ref{W>}.\;The
type of $R_i$ is denoted by ${\sf R}_i$.

{\it Step $2$.}

Arguing on the contrary, assume that $n=6$.\;Then \eqref{ineq} yields
$3\leqslant \nu_2(|G|)\leqslant 3$, i.e., $\nu_2(|G|)=\nu_2({\sf R})=3$.\;From this and Proposition \ref{l3}(ii) we deduce that
${\sf R}=a_1{\sf A}_1\stackrel{.}{+}a_2{\sf A}_2$, where
\begin{equation*}\label{rest1}
\begin{split}
a_1 +2a_2=6,\quad
a_1+a_2=3.
\end{split}
\end{equation*}
Hence $a_1=0$, $a_2=3$, i.e., ${\sf R}=3{\sf A}_2$.\;This and Table 1 yield
\begin{equation}\label{oe}
|G|=6^3,\; \exp(G)=6.
\end{equation}
On the other hand,  $|G|$ divides $6\exp(G)$ by Lemma 1(vii).\;This contra\-dicts \eqref{oe}.\;Hence
$n\neq 6$.

\vskip 1mm

{\it Step $3$.}

Arguing on the contrary, assume that $n=8$.\;Then
Lemma \ref{key}(vi) yields $\nu_3({\sf R})\leqslant 1$.\;Hence  by Pro\-po\-si\-tion  \ref{l3}(iii), the group $G$
contains a subgroup isomorphic to the Weyl group of a root system of type $(8-2){\sf A}_1$.\;This contradicts Propositi\-on\;\ref{mA1}(d), so we conclude that $n\neq 8$.

{\it Step $4$}

Arguing on the contrary, assume that $n=16$.\;Then
Lemma \ref{key}(vi) yields $\nu_3({\sf R})\leqslant 1$.\;Therefore, by Pro\-po\-si\-tion  \ref{l3}(iii), the group $G$ contains a subgroup isomorphic to the Weyl group of a root system of type $(16-2){\sf A}_1$. This contradicts Propositi\-on\;\ref{mA1}(d).\;Hence $n\neq 16$.
\end{proof}

\section {Proofs of Theorems \ref{repres} and \ref{R}}

\begin{proof}[Proof of Theorem {\rm \ref{repres}}]\

 (i)$\Rightarrow$(ii)\;  Assume that (i) holds.\;In view of Proposition
 \ref{124}, we have to show that if $n=4$, then ${\sf R}$ is
 either ${\sf A}_2\stackrel{.}{+}{\sf B}_2$ or
$ 2{\sf A}_1\stackrel{.}{+}
{\sf A}_2$.

So, let $n=4$.\;Then Lemma \ref{key}(v) (whose
notation we use) yields
\begin{equation}\label{44}
|\psi(G)|\!=\!1, 2, \;\mbox{or}\; 4.
\end{equation}
Next, we have $\nu_2(n)=2$, $\nu_3(n)=0$.\;Therefore,
Lem\-ma\;\ref{key}(vi) yields $\nu_2({\sf R})\!\leqslant\! 5$,
$\nu_3({\sf R})\!\leqslant\! 1$.\;Proposition \ref{l3}(iii) and Table 2 then infer that
$\sf R$ is one of the following root systems
\begin{gather*}
{\sf A}_4,
{\sf A}_1\stackrel{.}+{\sf A}_3,\,
{\sf A}_1\stackrel{.}+{\sf B}_3,\,
{\sf A}_1\stackrel{.}+{\sf C}_3,\,
2{\sf A}_1\stackrel{.}+{\sf A}_2,\\[-.5mm]
2{\sf A}_1\stackrel{.}+{\sf B}_2,\,
4{\sf A}_1,\,
{\sf A}_2\stackrel{.}+{\sf B}_2,\,
2{\sf A}_1\stackrel{.}+{\sf G}_2,\,
{\sf B}_2\stackrel{.}+{\sf G}_2
\end{gather*}

Assume that ${\sf R}={\sf A}_4$.\;Then $G$ is isomorphic to ${\rm Sym}_5$.\;Since the only proper normal subgroup of ${\rm Sym}_5$ is ${\rm Alt}_5$ (see, e.g., \cite[Thm.\,4.7]{Pa68}), and ${\rm Alt}_5$  is noncyclic, Lemma \ref{key} implies that $H$ is trivial, whence $|G|=5!$ divides $n=4$. This contradiction shows that, in fact, ${\sf R}\neq {\sf A}_4$.

Assume that ${\sf R}$ is one of the root systems
\begin{equation}\label{excl}
{\sf A}_1\stackrel{.}+{\sf B}_3,\,
{\sf A}_1\stackrel{.}+{\sf C}_3,\,
2{\sf A}_1\stackrel{.}+{\sf B}_2,\,
4{\sf A}_1,\,
2{\sf A}_1\stackrel{.}+{\sf G}_2,\,
{\sf B_2}\stackrel{.}+{\sf G}_2.
\end{equation}
Then, by Proposition \ref{l3}(iii), the group $G$
contains a subgroup isomorphic to the Weyl group of a root system of type
$4{\sf A}_1$.\;This contradicts Proposition \ref{mA1}(b). Hence, in fact,
${\sf R}$ is none of the root systems \eqref{excl}.

Finally, assume that ${\sf R}={\sf A}_1\stackrel{.}+{\sf A}_3$.\;Then $G$ is isomorphic to
${\rm Sym}_2\times {\rm Sym}_4$ and
\begin{equation}\label{orddd3}
|G|=2^4\cdot 3.
\end{equation}
Lemma \ref{key}(ii) and \eqref{orddd3} imply that,
respectively to \eqref{44}, we have  $|H|=2^4\cdot 3$,
$2^3\cdot 3$, or $2^2\cdot 3$, and, accordingly,
$\varphi(|H|)=2^4, 2^3$, or $2^2$.\;Since only the last
integer divides $4$, by Lemma \ref{key}(iv) we conclude that
$|\psi(G)|=4$.\;Hence $G$ contains a cyclic subgroup ${\rm mult}(H)$, whose
generator has order 12.\;But the maximum of orders of elements of
${\rm Sym}_2\times {\rm Sym}_4$ is $6$ (see, e.g., \cite[p.\;6]{Pa68}).\;This
contradiction
shows that, in fact, ${\sf R}\neq {\sf A}_1\stackrel{.}+{\sf A}_3$.

The proof of (i)$\Rightarrow$(ii) is now completed.

\vskip 2mm

(ii)$\Rightarrow$(i)\; If ${\sf R}\in \{{\sf A}_1, {\sf A}_2,
{\sf B}_2, {\sf G}_2, 2{\sf A}_1
\}$, then (i) follows from Proposi\-tion\;\ref{rank12}
and Definition \ref{def}.

Consider the case ${\sf R}= {\sf A}_2\stackrel{.}{+}{\sf B}_2$.

Let $K$ be the biquadratic field
$\mathbb Q(\sqrt{-3}, \sqrt{-1})$.\;Then
$$K=\mathbb Q({\sqrt{-3}})\otimes_{\mathbb Q}
\mathbb Q(\sqrt{-1}).$$
This equality determines the natural
homo\-mor\-phism
\begin{equation}\label{times}
\mathcal L
(\mathbb Q(\sqrt{-3}))\times \mathcal L(\mathbb Q
(\sqrt{-1}))\to \mathcal L(K),
\end{equation}
whose restriction to ${\rm Aut}(\mathbb
Q(\sqrt{-3}))\times {\rm Aut}(\mathbb Q(\sqrt{-1}))$
is an isomorphism with ${\rm Aut}(K)$ (see
\cite[Chap.\;VIII, \S1, Thm.\;5]{La65}).
The kernel of homomor\-phism \eqref{times} is 
$\{({\rm mult}(a), {\rm mult}(a^{-1}))
\mid a\in \mathbb Q^*\}$.

Let $R_1$ and $R_2$ be
respectively the realizations of type ${\sf A}_2$ in
$\mathbb Q(\sqrt{-3})$
and of type ${\sf B}_2$
in $\mathbb Q(\sqrt{-1})$ constructed in the proof of
Proposition \ref{rank12}.
Since $-1\notin W(R_1)$,
the restriction of homomorphism \eqref{times}
to the subgroup $W(R_1)\times
W(R_2)$ is an embedding.\;Therefore, its image is the
subgroup of $\mathcal L(K)$ isomorphic
to the Weyl group of a root system of type ${\sf A}_2
\stackrel{.}{+}{\sf B}_2$. This proves that (i) holds
if ${\sf R}= {\sf A}_2\stackrel{.}{+}{\sf B}_2$.

Now consider the case ${\sf R}=
{\sf A}_2\stackrel{.}{+} 2{\sf A}_1$.

If $R_3$ is a subset of  $R_2$,
which is a realization of type
$2{\sf A}_1$ in $K$,
then the restriction of homomorphism \eqref{times} to
$W({R_1})\times W({R_3})$
is the subgroup of $\mathcal L(K)$ isomorphic
to the Weyl group of a root system of type
${\sf A}_2\stackrel{.}{+} 2{\sf A}_1$.\;Thus (i)
holds if ${\sf R}$ is of this type.

This completes the proof of (ii)$\Rightarrow$(i)
and that of Theorem \ref{repres}.\end{proof}

\begin{proof}[Proof of Theorem {\rm \ref{R}}]

(i)$\Rightarrow$(ii)\;In view of Theorem \ref{repres}
and Definition \ref{def}, we have to show that
if $\sf R={\sf A}_2\stackrel{.}+ {\sf B}_2$ or ${\sf A}_2
\stackrel{.}+ 2{\sf A}_1$, then $\sf R$ admits no
realizations in the number fields.
Arguing on
the contrary, assume that
 this is not the case, i.e., $\sf R$ admits a
 realization in a number field $K$.

 The linear space $K$
over $\mathbb Q$ is then a direct sum of two
$2$-dimensional linear subspaces $L_1$ and $L_2$ such that
\begin{enumerate}[\hskip 4.2mm\rm(a)]
\item $L_i$ is the linear span of $R_i:=
R\bigcap L_i$ over $\mathbb Q$ for every $i$;
\item $R_1$ is a root system in $L_1$ of type ${\sf A}_2$;
\item $R_2$ is a root system in $L_2$ of
type
${\sf B}_2$ or
$2{\sf A}_1$;
\item $R=R_1\stackrel{.}+
R_2$.
\end{enumerate}
Let $\iota\colon {\rm GL}_\mathbb Q(L_1)\times {\rm GL}_\mathbb Q(L_2)
\!\hookrightarrow\! {\rm GL}_\mathbb Q(K)$
be the natural group embedding. Then
\begin{equation}\label{split}
W(R)=\iota(W(R_1))\times \iota(W(R_2)).
\end{equation}

In view of (b), the group $\iota(W(R_1))$ is
isomorphic to
${\rm Sym}_3$,
hence contains an element $z$ of order $3$.\;By \eqref{split}, we have
\begin{equation}\label{LK}
L_2\subseteq K^z.
\end{equation}

According to Lemma \ref{keyK}, there are
uniquely defined elements
$a\in K^*$ and $g\in {\rm Aut}(K)$ such that
$z={\rm mult}(a)g$.\;From \eqref{psi}
we infer
that ${\rm ord}(g)$ divides ${\rm ord}(z)\!=\!3$.\;Since
${\rm ord}(g)$ divides $|{\rm Aut}(K)|$, which,
in turn, divides $\dim_{\mathbb Q}(K)\!=\!4$, we conclude that
\begin{equation}\label{z}
z={\rm mult}(a).
\end{equation}
As ${\rm ord}(z)\neq 1$, we have $a\neq 1$.\;From this, \eqref{z}, and \eqref{mult1} we then conclude that
$K^z=0$ contrary to \eqref{LK}.\;This completes
the proof of (i)$\Rightarrow$(ii).

\vskip 1mm

(ii)$\Rightarrow$(i)\; This follows from
Proposition \ref{rank12}.
\end{proof}

\end{document}